\documentclass[12pt,reqno]{amsart}
\usepackage{amsmath,amssymb,amsfonts,amscd,latexsym,amsthm,mathrsfs,verbatim,mathabx}
\usepackage[colorlinks,linkcolor=black,citecolor=black]{hyperref}
\usepackage{hypbmsec}
\usepackage{bm}
\usepackage{tikz}
\usetikzlibrary{matrix,arrows,decorations.pathmorphing}
\usepackage[margin=1in]{geometry}
\newtheorem{lemma}{Lemma}
\newtheorem{theorem}{Theorem}

\newtheorem{defin}{Definition}
\newtheorem{prop}{Proposition}
\theoremstyle{remark}
\newtheorem*{remark}{\bf Remark}

\newtheorem{example}{\bf Example}

\newcommand{\Z}{\mathbb{Z}}

\newcommand{\R}{\mathbb{R}}

\newcommand{\N}{\mathbb{N}}

\usepackage{etoolbox}
\patchcmd{\section}{\scshape}{\bfseries}{}{}
\makeatletter
\renewcommand{\@secnumfont}{\bfseries}
\makeatother
\numberwithin{equation}{section}
\numberwithin{lemma}{section}
\numberwithin{theorem}{section}
\numberwithin{prop}{section}

\def\multsum{\mathop{\sum\cdots \sum}\limits}

\begin{document}

\title{Sums of Kloosterman sums over primes in an arithmetic progression}

\author{Alexander Dunn}
\address{Department of Mathematics, University of Illinois, 1409 West Green
Street, Urbana, IL 61801, USA}
\email{ajdunn2@illinois.edu}

\author{Alexandru Zaharescu}
\address{Department of Mathematics, University of Illinois, 1409 West Green
Street, Urbana, IL 61801, USA}
\address{Simon Stoilow Institute of Mathematics of the Romanian Academy, P.O. Box 1-764, RO-014700 Bucharest, Romania}
\email{zaharesc@illinois.edu}

\subjclass[2010]{Primary 11L07}
\keywords{Arithmetic progression, bilinear forms, Kloosterman sum, prime}

\maketitle

\begin{abstract}
For $q$ prime, $X \geq 1$ and coprime $u,v \in \N$ we estimate the sums
\begin{equation*}
\sum_{\substack{p \leq X \\ \substack p \equiv u \hspace{-0.25cm} \mod{v} \\ p \text{ prime}}} \text{Kl}_2(p;q),
\end{equation*}
where $\text{Kl}_2(p;q)$ denotes a normalised Kloosterman sum with modulus $q$. This is a sparse analogue of a recent theorem due to Blomer, Fouvry, Kowalski, Michel and Mili\'{c}evi\'{c} showing cancellation amongst sums of Kloosterman sums over primes in short intervals. We use an optimisation argument inspired by Fouvry, Kowalski and Michel. Our argument compares three different bounds for bilinear forms involving Kloosterman sums. The first input in this method is a bilinear bound we prove using uniform asymptotics for oscillatory integrals due to Kiral, Petrow and Young. In contrast with the case when the sum runs over all primes, we exploit cancellation over a sum of stationary phase integrals that result from a Voronoi type summation. The second and third inputs are deep bilinear bounds for Kloosterman sums due to Fouvry--Kowalski--Michel and Kowalski--Michel--Sawin. 
\end{abstract}

\section{Introduction}
Cancellation among exponential sums defined over finite fields plays a central role in number theory. A key example is Weil's bound for exponential sums of rational functions, see \cite[p.~1]{FKM}. It is natural to ask whether cancellation persists when such a summation is restricted to primes, or primes in short intervals. For instance, one can see \cite[Th\'{e}or\`{e}me~1.1]{FM}. 

It is also natural to ask the same question for other functions defined on $\mathbb{F}_q$, where $q$ is a prime. Important examples are the hyper--Kloosterman sums in $m-1$ variables introduced by Deligne, and later studied by Katz \cite{K}. For $(n,q)=1$, these are defined by  
\begin{equation} \label{hyper}
\text{Kl}_m(n;q)=\frac{1}{q^{\frac{m-1}2}} \multsum_{\substack{x_1\cdots
  x_m=n \\ x_i \in \mathbb{F}_q}} e \Big(\frac{x_1+\cdots+x_m}{q} \Big)
\end{equation} 
These fall into a general class of functions called \emph{trace weights} \cite{FKM3}, studied extensively by Fouvry, Kowalski and Michel \cite{FKM}. These authors prove general results asserting cancellation amongst trace weights over primes \cite[Theorem~1.5]{FKM}. As a consequence they obtain results for sums of classical Kloosterman sums ($m=2$ in \eqref{hyper}) \cite[Corollary~1.13]{FKM}.  Blomer, Fouvry, Kowalski, Michel and Mili\'{c}evi\'{c} recently improved \cite[Corollary~1.13]{FKM} in a special issue of the Proceedings of the Steklov Institute of Mathematics dedicated in the memory of I.~M~ Vinogradov \cite{BFKPM, BFKPM2}.

\begin{theorem} \cite[Theorem~1.7]{BFKPM2} \label{motivthm}
Let $q$ be a prime number. For every $1 \leq X \leq q$ and $\varepsilon>0$, we have 
\begin{equation} \label{motiv}
\sum_{\substack{p \leq X  \\ p \emph{ prime}}} \emph{Kl}_2(p;q) \ll_{\varepsilon} q^{\frac{1}{6}+\varepsilon} X^{\frac{7}{9}}.
\end{equation}
\end{theorem}

The goal of this paper is to obtain a sparse analogue of Theorem \ref{motivthm}. We concentrate on the case when $p$ runs over primes in an arithmetic progression for a large range of moduli.

\begin{theorem} \label{nonsmooththm}
Let $q$ be a prime number and $u,v \in \mathbb{N}$ such that $(u,v)=1$ and $1 \leq v \leq q^{\frac{1}{100}}$. Then for all $1 \leq X \leq q$ and $\varepsilon>0$ we have
\begin{equation} \label{tradstat}
\sum_{\substack{p \leq X \\ p \equiv u \hspace{-0.25cm} \mod v \\
p \emph{ prime} }}   \emph{Kl}_2(p,q) \ll_{\varepsilon}  q^{\frac{11}{192}+\varepsilon} X^{\frac{15}{16}}.
\end{equation}
\end{theorem}
We offer some conditional and unconditional remarks about the various ranges where Theorem \ref{nonsmooththm} is non-trivial. Naturally, this depends on the size
of the set of primes over which one sums on the left side of \eqref{tradstat}, which in turn may be influenced by potential Siegel zeros. Let 
\begin{equation*}
\psi(X;v,u):=\sum_{\substack{n \leq X \\ n \equiv u \hspace{-0.2cm} \mod{v}}} \Lambda(n).
\end{equation*}
\begin{itemize}
\item For $A>0$ the Siegel--Walfisz Theorem \cite[Corollary~5.29]{IK} asserts that 
\begin{equation*}
\psi(X;v,u)=\frac{X}{\phi(v)}+O_A \big(X (\log X)^{-A} \big)
\end{equation*}
for all $X \geq 2$ and $v \geq 1$ such that $(u,v)=1$. In light of the Weil bound $|\text{Kl}_2(m;q)| \leq 2$, we see that Theorem \ref{nonsmooththm} is non-trivial in the ranges $X \gg q^{\frac{11}{12}+\varepsilon}$ and $v \ll (\log X)^A$. For instance, when $v$ is fixed and $X=q$, we gain a factor of $q^{\frac{1}{192}-\varepsilon}$ over the trivial result obtained by applying the Weil bound.
\item For $A>0$, a well known consequence of the Bombieri--Vinogradov Theorem \cite[Theorem~17.1]{IK} asserts
\begin{equation} \label{bombvin}
\sum_{1 \leq v \leq X^{\frac{1}{2}-\varepsilon}} E(X;v) \ll_{A} \frac{X}{\log^A X}, \quad E(X;v):=\max_{(u,v)=1} \bigg | \psi(X;v,u)-\frac{X}{\phi(v)} \bigg |. 
\end{equation}
In the following suppose that $q^{\frac{11}{12}+\varepsilon} \ll X \leq q$. Then \eqref{bombvin} implies that
\begin{equation*}
\psi(X;v,u)=\frac{X}{\phi(v)}+O \Big( \frac{X^{\frac{15}{16}} q^{\frac{11}{192}+\varepsilon}}{\log^{10} X} \Big)
\end{equation*}
holds for all moduli $1 \leq v \leq X^{\frac{1}{2}-\varepsilon}$ with at most 
\begin{equation*}
O \Big( \frac{X^{\frac{1}{16}} q^{-\frac{11}{192}-\varepsilon}}{\log^{100} X} \Big)
\end{equation*}
exceptions. We see that Theorem \ref{nonsmooththm} is non--trivial for $X \gg q^{\frac{11}{12}+\varepsilon}$ and almost all moduli  $v \ll X^{\frac{1}{16}} q^{-\frac{11}{192}-\varepsilon}$.
\item The Generalised Riemann hypothesis for $L(s,\chi)$ with $\chi \pmod{q}$ implies that 
\begin{equation} \label{countprog}
\psi(X;v,u)=\frac{X}{\phi(v)}+O \big(X^{\frac{1}{2}} (\log X)^2 \big),
\end{equation}
where the implied constant is absolute \cite[Equation~17.4]{IK}. In light of the Weil bound (stated above), we see that Theorem \ref{nonsmooththm} is non--trivial when $X \gg q^{\frac{11}{12}+\varepsilon}$ and all moduli $v \leq X^{\frac{1}{16}} q^{-\frac{11}{192}-\varepsilon}$ by \eqref{countprog}. 
\end{itemize} 

After some adjustments, an application of Heath--Brown's identity \cite{HB} for the Von Mangoldt function reduces the proof of Theorem \ref{nonsmooththm} to bounding bilinear forms of various lengths involving Kloosterman sums. Three different bounds for such bilinear forms will be invoked. The utility of each bound will depend on the ranges of summation of the bilinear form. 

For bilinear forms whose summation variables lie a certain range, we detect the congruence condition $\mathbf{1}_{u \hspace{-0.15cm} \mod{v}}$ additively. This transfers the task at hand to bounding a bilinear form whose summand is a Kloosterman sum weighted by an oscillatory smooth function. The tempered Voronoi summation formula of Deshouillers and Iwaniec turns this bilinear form into a sum of Fourier transforms. This summation formula is given in Section \ref{voronoisec}.

 The new key idea in our work is to exploit cancellation from the sum of transforms that come out of this tempered Voronoi summation.  We appeal to the uniform asymptotics for stationary phase integrals given by Kiral, Petrow and Young \cite{KPY} to compute asymptotic main terms for these Fourier transforms. A summary of their work can be found in Section~\ref{inertos}. Section \ref{weightfn} sets up the families of weight functions that will be used. Furthermore, Sections \ref{nonstatforms} and \ref{statforms} handle the cases when the Fourier transforms of these functions are non-stationary and stationary phase respectively. The application to bounding the bilinear forms involving Kloosterman sums in question can be found in Section \ref{bilinos}.

For other ranges we detect the congruence condition $\mathbf{1}_{u \hspace{-0.15cm} \mod{v}}$ multiplicatively.  In this case we can appeal to existing bilinear bounds for Kloosterman sums due to Fouvry--Kowalski--Michel \cite{FKM} and Kowalski--Michel--Sawin \cite{KMS}. These bounds are recorded in Section \ref{bilinbounds}. Shparlinski and Zhang \cite{IZ} have also proved bounds for various bilinear forms with Kloosterman sums.

The proof of a smoothed version of Theorem \ref{nonsmooththm} is given in Section \ref{smoothpropsec}. The main idea behind this proof is a combinatorial optimisation of the three bounds discussed above. We modify an argument of Fouvry, Kowalski and Michel appearing in \cite{FKM}, and subsequently in \cite{BFKPM}. The proof of Theorem \ref{nonsmooththm} appears in Section \ref{mainthms}, and will follow from a careful choice of parameters in the work in Section \ref{smoothpropsec}.

\section{Tempered Voronoi summation} \label{voronoisec}
Let $q$ be a prime number and let $K: \mathbb{Z} \rightarrow \mathbb{C}$ be a $q$-periodic function. The \emph{normalised Fourier transform} of $K$ is the $q$-periodic function on $\mathbb{Z}$ defined by 
\begin{equation*}
\widehat{K}(h)=\frac{1}{\sqrt{q}} \sum_{n \hspace{-0.25cm} \mod{q}} K(n) e \Big( \frac{hn}{q} \Big).
\end{equation*}
The \emph{Voronoi transform} of $K$ is the $q$-periodic function on $\Z$ defined by  
\begin{equation*}
\widecheck{K}(n)=\frac{1}{\sqrt{q}} \sum_{\substack{h \hspace{-0.25cm} \mod{q} \\ (h,q)=1 }} \widehat{K}(h) e \Big( \frac{\bar{h} n}{q} \Big).
\end{equation*}

We will use the tempered Voronoi summation formula of Deshouillers and Iwaniec as given in \cite[Prop~2.2]{FKM2}.

\begin{prop} \cite[Prop~2.2]{FKM2} \label{temper}
Let $q$ be a prime number and let $K: \mathbb{Z} \rightarrow \mathbb{C}$ be a $q$-periodic function, and $G$ be a smooth function on $\mathbb{R}^2$ with compact support and Fourier transform denoted by $\widehat{G}$. Then we have 
\begin{equation} \label{voronoi}
\sum_{m,n \in \mathbb{Z}} K(mn) G(m,n)=\frac{\widehat{K}(0)}{\sqrt{q}} \sum_{m,n \in \Z} G(m,n)+\frac{1}{q} \sum_{m,n \in \Z} \widecheck{K}(mn) \widehat{G} \Big( \frac{m}{q}, \frac{n}{q} \Big).
\end{equation}
\end{prop}

If $K$ is a multiplicatively shifted Kloosterman sum, then $\widecheck{K}$ is a normalised Dirac delta function. For $(a,q)=1$ and $K(n):=\text{Kl}_2(an;q)$, an immediate computation shows that 
\begin{equation*}
\widehat{K}(h)=\begin{cases}
0 & \text{ if } q \mid h \\
e \big({-\frac{a \overline{h}}{q}} \big) & \text{ if } q \nmid h,
\end{cases}
\end{equation*}
and 
\begin{equation*}
\widecheck{K}(n)=\begin{cases}
\frac{q-1}{q^{1/2}} & \text{ if } n \equiv a \pmod{q} \\
-\frac{1}{q^{1/2}} & \text{ otherwise}.
\end{cases}
\end{equation*}

\section{Bilinear bounds} \label{bilinbounds}
In this section we record two bilinear bounds for Kloosterman sums that we will use in conjunction with Proposition \ref{bilinprop}. The first bound is a smoothed version of a bound due to Fouvry, Kowalski and Michel \cite{FKM}.  This smoothed version appears in \cite{BFKPM2}. For $\alpha \in \mathbb{C}^R$, let $\| \alpha \|_2$ denote the usual $\ell^2$--norm.

\begin{theorem} \cite[Theorem~1.17]{FKM} and \cite[Proposition~2.3]{BFKPM2} \label{FKM}
Let $q$ be a prime number and $1 \leq M,N \leq q$ and $(\alpha_m)$ and $(\beta_n)$ be sequences of complex numbers supported in $[M,2M]$ and $[N,2N]$ respectively. Let $Q=1$ and $W$ be the constant function 1, or $Q \geq 1$ and $W$ be a smooth function satisfying \eqref{smooth} below. Then for every $\varepsilon>0$ we have 
\begin{equation*}
\sum_{m,n} \alpha_m \beta_n \emph{Kl}_2(mn;q) W \Big( \frac{mn}{Y} \Big) \ll_{\varepsilon} \| \alpha \|_2 \| \beta \|_2 (MN)^{\frac{1}{2}} \Big( \frac{1}{M}+Q \frac{q^{\frac{1}{2}+\varepsilon}}{N} \Big)^{\frac{1}{2}}.
\end{equation*}
\end{theorem}

The second bound is a recent result due to Kowalski, Michel and Sawin \cite{KMS}. Using deep methods from algebraic geometry they prove general bounds for hyper--Kloosterman sums in the \emph{Polya--Vinogradov} range. This is when the parameters $M$ and $N$ are both close to $q^{\frac{1}{2}}$. We state their result in a special case.

\begin{theorem} \cite[Theorem~1.1]{KMS} \label{KMS}
Let $q$ be a prime number. Let $a$ be an integer coprime to $q$.  Let $M,N \geq 1$ be real numbers such that
\begin{equation*}
1 \leq M \leq q^{1/4} N \quad \text{and} \quad q^{1/4}<MN<q^{5/4}.
\end{equation*}
Let $\mathcal{N} \subset [1,q-1]$ be an interval of length $\lfloor N \rfloor$, and let $\boldsymbol{\alpha}=(\alpha_m)_{m \leq M}$ and $\boldsymbol{\beta}=(\beta_n)_{n \in \mathcal{N}}$ be sequences of complex numbers. For any $\varepsilon>0$, we have 
\begin{equation} \label{KMSbd}
\sum_{m,n} \alpha_m \beta_n \emph{Kl}_2(amn;q) \ll_{\varepsilon} q^{\varepsilon} \| \alpha \|_2 \|\beta\|_2 (MN)^{\frac{1}{2}} \big(M^{-\frac{1}{2}}+(MN)^{-\frac{3}{16}} q^{\frac{11}{64}} \big).
\end{equation}
\end{theorem}

We will need a smoothed version of Theorem \ref{KMS}. Let $Q \geq 1$, $\varepsilon>0$ and let $W$ be a smooth function satisfying \eqref{smooth} below. Suppose $X \geq 1$ and $(\alpha_m)$ and $(\beta_n)$ are supported on intervals $[M,2M]$ and $[N,2N]$ respectively. We will need an estimate for
\begin{equation} \label{smoothbilin}
\sum_{\substack{m,n}} \alpha_m \beta_n \text{Kl}_2(amn;q)  W \Big(\frac{mn}{X} \Big). 
\end{equation}
Without loss of generality we may assume $MN/8<X<8MN$, otherwise the sum in \eqref{smoothbilin} is empty. Applying Abel summation to \eqref{smoothbilin} yields 
\begin{equation} \label{smthbd}
\sum_{\substack{m,n}} \alpha_m \beta_n \text{Kl}_2(amn;q)  W \Big(\frac{mn}{X} \Big) \ll_{\varepsilon} (q^{\varepsilon} Q)^2 \| \alpha \|_2 \|\beta\|_2 (MN)^{\frac{1}{2}} \Big(M^{-\frac{1}{2}}+(MN)^{-\frac{3}{16}} q^{\frac{11}{64}} \Big).
\end{equation}

\section{Inert functions and oscillatory integrals} \label{inertos}
We make use of the framework of inert functions and oscillatory integrals set out by Kiral, Petrow and Young in \cite{KPY}. Their work is a multivariable generalisation of Blomer, Khan and Young \cite[Proposition~8.2 and Corollary~8.3]{BKY}. We recall their definition of inert functions. These are families of functions that satisfy certain derivative bounds. Let $\mathcal{F}$ be an index set and $X:=X_T: \mathcal{F} \rightarrow \R_{\geq 1}$ be a function of $T \in \mathcal{F}$.
\begin{defin}
A family $\{w_T\}_{T \in \mathcal{F}}$ of smooth functions supported on a product of dyadic intervals in $\R_{>0}$ is called $X$-inert if for each $j=(j_1,j_2,\ldots,j_d) \in \Z_{\geq 0}^d$ we have 
\begin{equation*}
C(j_1,\ldots,j_d):=\sup_{T \in \mathcal{F}} \sup_{(x_1,\ldots,x_d) \in \R_{>0}^d} X_T^{-j_1-\cdots-j_d} \Big | x_1^{j_1} \cdots x_d^{j_d} w_T^{(j_1,\ldots,j_d)}(x_1,\ldots,x_d)  \Big |<\infty.
\end{equation*}
\end{defin}
The notion of $X$-inertness measures the uniformity of \emph{flatness} of the functions $w_T$ as we move across the family $\mathcal{F}$.  The sequence of constants $C(j_1,\ldots,j_d)$ is abbreviated by $C_{\mathcal{F}}$. We present one key example whose idea is heavily used throughout this paper.
\begin{example}[Dilation]
Let $w$ be a fixed smooth function on $[1,2]^d$ and define 
\begin{equation*} 
w_{X_1,\ldots,X_d}(x_1,\ldots,x_d)=w \Big(\frac{x_1}{X_1}, \cdots, \frac{x_d}{X_d}  \Big).
\end{equation*} 
Then with $\mathcal{F}=\{T=(X_1,\ldots,X_d) \in \R_{>0}^d \}$, the family $\{ w_T\}_{T \in \mathcal{F}}$ is $1$-inert.
\end{example}

The stationary phase method appears in many references and gives an asymptotic expansion for a Fourier integral of the form $\int_{-\infty}^{\infty} a(x) e^{i \lambda \phi(x)} dx$, where $a$ is smooth with compact support and $\phi^{\prime}(x_0)=0$ for a unique $x_0$ in the support of $a$. For example, see \cite[Theorem~7.7.5]{Hor}, \cite[Ch. VIII Proposition 3]{St} and \cite[Theorem~3.11]{Zwor}. As pointed out in \cite{KPY}, such integrals are often insufficient for applications in analytic number theory. 

In our case, we wish to analyse the Fourier transform of a certain family of oscillatory weight functions. This analysis needs to be uniform with respect to some auxiliary parameters. This type of situation naturally occurs after an application of a summation formula, in our case, after an application of Proposition \ref{temper}.  Integrals arising from this process do not usually have a phase of the form $\lambda \phi(x)$. In special cases one may try reduce the phase to a function of the form $\lambda \phi(x)$ using an ad-hoc change of variables. However, this is often not feasible in practice. 

The following recent theorem due to Kiral, Petrow and Young gives an asymptotic for stationary phase integrals that is uniform with respect to many parameters.

\begin{theorem} \cite[Main Theorem]{KPY} \label{KPY}
Let $\mathcal{F}$ be an index set and suppose $\{w_T\}_{T \in \mathcal{F}}$ is an $X$-inert family of functions in the variables $t_1,\ldots,t_d$, supported on $t_1 \asymp Z$ and $t_i \asymp X_i$ for $i=2,\ldots,d$. Suppose that on the support of $w_T$, $\phi=\phi_T$ satisfies 
\begin{equation} \label{phibd}
\frac{\partial^{a_1+\cdots+a_d}}{\partial t_1^{a_1} \cdots \partial t_d^{a_d}} \phi(t_1,\ldots,t_d) \ll_{\mathcal{C}_{\mathcal{F}}} \frac{Y}{Z^{a_1}} \frac{1}{X_2^{a_2} \cdots X_d^{a_d} }
\end{equation}
for all $a_1,\ldots, a_d \in \mathbb{N}$. Suppose that 
\begin{equation} \label{secondstat}
\phi^{\prime \prime}(t_1,\ldots,t_d) \gg Y/Z^2 \quad  \text{for all } t_1,t_2,\ldots,t_d  \text{ in the support of } w_T,
\end{equation}
where the derivatives are taken with respect to $t_1$, and that there exists $t_0 \in \mathbb{R}$ such that $\phi^{\prime}(t_0)=0$ (note that $t_0$ is necessarily unique). Suppose that $Y/X^2 \geq R \geq 1$. Then 
\begin{equation} \label{integral}
\int_{\R} e \big( \phi(t_1,\ldots,t_d) \big) w_T(t_1,\ldots,t_d) dt_1=\frac{Z}{\sqrt{Y}} \mathcal{W}_T(t_2,\ldots,t_d) e \big( \phi(t_0,t_2,\ldots,t_d) \big) +O_A(Z R^{-A}),
\end{equation}
for some family of $X_T$-inert functions $\mathcal{W}_T$, and where $A>0$ is arbitrarily large. The implied constant in \eqref{integral} depends only on $A$ and $C_{\mathcal{F}}$.
\end{theorem}

\begin{remark}
It is clear from the definition of $\mathcal{W}_T$ on \cite[pg.~7]{KPY} that $\mathcal{W}_T$ has the same support as $w_T$ in the variables $t_i$ for $i=2,\ldots,d$.
\end{remark}

The main advantage of using Theorem \ref{KPY} to compute $\widehat{G}$ on the right side of \eqref{voronoi} is that we are subsequently able to exploit further cancellation in the summation on $m$ and $n$. In particular, we do this for ranges of $m$ and $n$ where $\widehat{G}(\frac{m}{q},\frac{n}{q})$ is a stationary phase integral. In other ranges it suffices to integrate by parts.

\section{Oscillatory weight functions} \label{weightfn}
For $\varepsilon>0$ given, we will index our family of smooth oscillatory weight functions by elements of the set
\begin{equation} \label{T}
\mathcal{T}_1(\varepsilon):=\Big \{(c,q,M,N,Q) \in \R_{>0} \times \R_{\geq 1} \times \R_{\geq \frac{1}{2}}^2 \times \R_{\geq 1}: \frac{cMN}{(q^{\varepsilon} Q)^2} \geq 1 \Big \}.
\end{equation}
For $T \in \mathcal{T}_1(\varepsilon)$, let $W_i: [0,\infty] \rightarrow \R$ be a smooth function such that
\begin{equation} \label{smooth}
\text{supp} \hspace{0.05cm} W_i \subset [0.95,1.05] \quad \text{and} \quad W_i^{(j)}(x) \ll_{j,\varepsilon} (q^{\varepsilon} Q)^j \text{ for any } x \geq 0, j \geq 0.
\end{equation}

\begin{remark}
Let $\varepsilon>0$ and suppose $W_i^{(j)}(x) \ll_{j,\varepsilon} (q^{\varepsilon} Q)^j$ holds above with system of constants $\mathcal{C}(\varepsilon):=\{C(j,\varepsilon): j \geq 0 \}$. All implied constants in this paper will depend on $\varepsilon$ and $\mathcal{C}(\varepsilon)$. We will omit this from subscripts in the proofs. On occasion when an implied constant depends on another auxiliary parameter, we will indicate this.
\end{remark}

For $T \in \mathcal{T}_1(\varepsilon)$ define 
\begin{equation} \label{G}
G_T(x,y):=W_1 \Big(\frac{x}{M} \Big) W_2 \Big(\frac{y}{N} \Big) e (cxy).
\end{equation}

For $T \in \mathcal{T}_1(\varepsilon)$ and $U \geq 1$, define
\begin{equation} \label{H}
H_{T,U}(x,y):=W_1 \Big(\frac{x}{M} \Big) W_2 \Big(\frac{y}{N} \Big) W_3 \Big(\frac{xy}{U} \Big) e (cxy).
\end{equation}

\section{Non-stationary transforms} \label{nonstatforms}
In the following lemma we estimate the Fourier transforms $\widehat{G}_T \big(m/q,n/q \big)$ and $\widehat{H}_{T,U} \big(m/q,n/q \big)$ for those values of $m$ and $n$ that correspond to a non-stationary phase integral.
\begin{lemma} \label{oscill2}
Let $\varepsilon>0$ and $T \in \mathcal{T}_1(\varepsilon)$ be as in \eqref{T}. For $T \in \mathcal{T}_1(\varepsilon)$ and $U \geq 1$, let $W_i$, $G_T$ and $H_{T,U}$ be as in \eqref{smooth}, \eqref{G} and \eqref{H} respectively. Then 
\begin{equation} \label{nonstat}
\widehat{G}_T \big(\frac{m}{q},\frac{n}{q} \big)  \ll \begin{cases}
 \frac{(q^{1+\varepsilon} Q)^{2+\varepsilon}}{|mn|^{1+\varepsilon}   } &   \emph{ if } |n| \geq 1.1cqM \emph{ and } |m| \geq 1.1cqN  \\[0.5em] 
 M N \big(\frac{q^{1+\varepsilon} Q}{|n|N} \big)^{2+\varepsilon} &   \emph{ if } |n| \geq 1.1cqM \emph{ and }  |m| \leq 1.1cqN  \\[0.5em] 
 M N \big(\frac{q^{1+\varepsilon} Q}{|m|M} \big)^{2+\varepsilon}  &  \emph{ if } |n| \leq 1.1cqM \emph{ and }   |m| \geq 1.1cqN \\[0.5em]
\frac{(q^{\varepsilon} Q)^2}{c^2 MN}  &   \emph{ if }   [-1.1cqM \leq n \leq 0.9cqM \emph{ and } |m| \leq 1.1cqN]  \\[0.5em]
 &  \emph{ or }  [|n| \leq 1.1cqM \emph{ and }  -1.1cqN \leq m \leq 0.9cqN]. \\
\end{cases}
\end{equation}
The bound \eqref{nonstat} is also satisfied by $\widehat{H}_{T,U}$. The implied constant depends only on $\varepsilon>0$ and $\mathcal{C}(\varepsilon)$.
\end{lemma}

\begin{proof}
This argument amounts to repeated integration by parts. We first write $\widehat{G}_T$ as follows
\begin{equation} \label{order}
\widehat{G}_T \Big(\frac{m}{q},\frac{n}{q} \Big)=  \int_{\R} W_1 \Big( \frac{u}{M} \Big) e \Big({-\frac{mu}{q}} \Big) \int_{\R} W_2 \Big( \frac{v}{N} \Big)  e \Big({cuv-\frac{nv}{q}}  \Big) dv du.
\end{equation}
If $|n| \geq  1.1cqM$, then applying integration by parts $A$ times with respect to $v$ yields 
\begin{equation} \label{range1}
\widehat{G}_T \Big(\frac{m}{q},\frac{n}{q} \Big) \ll_{A} M N \Big(\frac{q^{1+\varepsilon} Q}{N |n|} \Big)^{A} \quad \text{for all } m.
\end{equation}
If $-1.1cqM \leq n \leq 0.9cqM$, then by a similar argument we obtain 
\begin{equation} \label{range2}
\widehat{G}_T \Big(\frac{m}{q},\frac{n}{q} \Big) \ll_A MN \Big( \frac{q^{\varepsilon} Q}{cMN} \Big)^A  \quad \text{for all } m.
\end{equation}

For $|m| \geq  1.1cqN$, one can interchange the order of integration in \eqref{order} to obtain 
\begin{equation} \label{range3}
\widehat{G}_T \Big(\frac{m}{q},\frac{n}{q} \Big) \ll_{A} M N \Big(\frac{q^{1+\varepsilon} Q}{M |m|} \Big)^{A} \quad \text{for all } n.
\end{equation}
Similarly for $-1.1cqN \leq m \leq 0.9 cqN$ we obtain
\begin{equation} \label{range4}
\widehat{G}_T \Big(\frac{m}{q},\frac{n}{q} \Big) \ll_{A} M N \Big(\frac{q^{\varepsilon} Q}{cMN} \Big)^{A} \quad \text{for all } n.
\end{equation}
Note that \eqref{range1}--\eqref{range4} hold for all $A \geq 1$ since 
\begin{equation} \label{mono}
\min (C,D) \leq C^{\alpha} D^{1-\alpha}, \quad 0<\alpha<1 \quad \text{and} \quad C,D>0.
\end{equation}

The first line of \eqref{nonstat} follows from applying \eqref{mono} ($\alpha=1/2$) to \eqref{range1} and \eqref{range3} with choice $A=2+\varepsilon$. The last line of \eqref{nonstat} follows directly from \eqref{range2} and \eqref{range4} with the choice $A=2$. The second and third lines of \eqref{nonstat} follow from \eqref{range1} and \eqref{range3} respectively with choice $A=2+\varepsilon$.

Without loss of generality we may assume $MN/8<U<8MN$, otherwise $H_{T,U} \equiv 0$. Consider the family of weight functions 
\begin{equation*}
W_T(x,y):=W_1 \Big( \frac{x}{M}\Big) W_3 \Big(\frac{xy}{U} \Big) \quad \Big( \text{resp. }  W_2 \Big( \frac{y}{N}\Big) W_3 \Big(\frac{xy}{U} \Big)  \Big).
\end{equation*}
These would now appear in the innermost integral occurring in \eqref{order}. We have
\begin{equation} \label{triple3}
\frac{\partial^{i+j} W_T(x,y)}{\partial x^i \partial y^j} \ll_{i,j} (q^{\varepsilon} Q)^{i+j} x^{-i} y^{-j}.
\end{equation}
Let 
\begin{equation*}
X_{T,U}:=q^{\varepsilon}Q.
\end{equation*}
Thus $W_{T,U}:\mathbb{R}^2 \rightarrow \mathbb{R}$ is a family of $X_{T,U}$-inert functions with system of constants determined by $\mathcal{C}(\varepsilon)$. Thus the argument above can be followed to establish \eqref{nonstat} for $H_{T,U}$.
\end{proof}

\section{Stationary transforms} \label{statforms}

We now give an asymptotic main term for $\widehat{G}_T(m/q,n/q)$ and  $\widehat{H}_{T,U}(m/q,n/q)$ for those values of $m$ and $n$ that correspond to a stationary phase integral. We will eventually exploit further cancellation from sums of these main terms in Section \ref{bilinos}.

\begin{lemma} \label{oscill}
Let $\varepsilon>0$ and  $\mathcal{T}_1(\varepsilon)$ be as in \eqref{T}. For $T:=(c,q,M,N,Q) \in \mathcal{T}_1(\varepsilon)$ and $U \geq 1$, let $W_i$, $G_T$ and $H_{T,U}$ be as in \eqref{smooth}, \eqref{G} and \eqref{H} respectively. Let 
\begin{equation*}
X_{T},X_{T,U}:=q^{\varepsilon}Q.
\end{equation*}
Then there exists $X_{T}$-inert (resp. $X_{T,U}$-inert) functions $\mathcal{W}_{T}(m,n)$ and $\mathcal{W}_{T,U}(m,n)$ such that for 
\begin{equation} \label{statrange}
(m,n) \in [0.9 cqN,1.1cqN] \times [0.9 cqM,1.1cqM],
\end{equation}
we have 
\begin{equation} \label{equal}
\widehat{G}_T \Big(\frac{m}{q},\frac{n}{q} \Big)=\frac{1}{c} \mathcal{W}_T(m,n) e \Big({-\frac{mn}{cq^2}} \Big)+O \Big( \frac{(q^{\varepsilon} Q)^6}{c^3 (MN)^2} \Big),
\end{equation}
and 
\begin{equation} \label{equal2}
\widehat{H}_{T,U} \Big(\frac{m}{q},\frac{n}{q} \Big)=\frac{1}{c} \mathcal{W}_{T,U}(m,n) e \Big({-\frac{mn}{cq^2}} \Big)+O \Big( \frac{(q^{\varepsilon} Q)^6}{c^3 (MN)^2} \Big).
\end{equation}
Both $\mathcal{W}_T$ and $\mathcal{W}_{T,U}$ have support contained in
\begin{equation} \label{mnsupport}
(m,n) \in [0.85 cqN,1.15cqN] \times [0.85 cqM,1.15cqM].
\end{equation}
All implied constants depend only on $\varepsilon>0$ and $\mathcal{C}(\varepsilon)$.
\end{lemma}

\begin{proof}
When $T \in \mathcal{T}_1(\varepsilon)$ and $(m,n)$ satisfies \eqref{statrange}, the integral $\widehat{G}_T(m/q,n/q)$ is stationary phase. 
\subsection{Technical set-up}
First,
\begin{equation} \label{fourier}
\overline{\widehat{G}_T \Big( \frac{m}{q},\frac{n}{q} \Big)}= \int_{\R^2} W_1 \Big( \frac{u}{M} \Big) W_2 \Big( \frac{v}{N} \Big) e \Big({-cuv+\frac{mu}{q}+\frac{nv}{q}}  \Big) du dv.
\end{equation}
Viewing $v$ as fixed, we make the change of variable $u \mapsto u N/v $ in the right side of \eqref{fourier} to obtain 
\begin{equation} \label{fourier2}
\overline{\widehat{G}_T \Big( \frac{m}{q},\frac{n}{q} \Big)}= \int_{\R^2} \frac{N}{v}  W_1 \Big(\frac{uN}{vM} \Big) W_2 \Big( \frac{v}{N} \Big) e \Big({-cNu+\frac{muN}{vq}+\frac{nv}{q} } \Big) du dv.
\end{equation}
This change of variable is made to de-linearize the phase function in \eqref{fourier} so that the method of stationary phase may be applied with respect to the $v$ variable.

Before computing the right side of \eqref{fourier2} we make the following modification for technical reasons. Let $W_4 :\R \rightarrow \R$ be any smooth function such that 
\begin{align*}
\text{supp} \hspace{0.05cm} W_4 \subset \big[0.85,1.15 \big],  &\quad 0 \leq W_4 \leq 1, \quad  W_4(x)=1 \quad \text{for} \quad 0.89 \leq x \leq 1.11, \\
\quad \text{and}  \quad & W_4^{(j)}(x) \ll_{\varepsilon,j} (q^{\varepsilon} Q)^j   \quad \text{for any}  \quad x \geq 0, j \geq 0,
\end{align*}
with the same system of constants $\mathcal{C}(\varepsilon)$. Then for $(m,n)$ satisfying \eqref{statrange}, we have 
\begin{multline} \label{modify}
\overline{\widehat{G}_T \Big( \frac{m}{q},\frac{n}{q} \Big)}= \int_{\R} W_4 \Big( \frac{u}{M} \Big) \int_{\R} \frac{N}{v}  W_1 \Big(\frac{uN}{vM} \Big)  \\
\times W_2 \Big( \frac{v}{N} \Big) W_4 \Big( \frac{m}{cqN} \Big) W_4 \Big( \frac{n}{cqM} \Big) e \Big({-cNu+\frac{muN}{vq}+\frac{nv}{q} } \Big) dv du.
\end{multline}
\subsection{Integration with respect to $v$}
Set 
\begin{align}
\Phi_{T}(u,v;m,n)&:=\Phi(u,v;m,n)=-cNu+\frac{muN}{vq}+\frac{nv}{q} \nonumber \\
w_T(u,v;m,n)&:=\frac{N}{v} W_1 \Big( \frac{uN}{vM} \Big) W_2 \Big( \frac{v}{N} \Big) W_4 \Big( \frac{m}{cqN} \Big) W_4 \Big( \frac{n}{cqM} \Big)  \label{inerty}.
\end{align} 
Observe that $\Phi_T$ satisfies \eqref{phibd} with $Y=cMN$ with respect to the variables $u,v,m$ and $n$ on their respective supports. Furthermore, $w_T$ is $X_T$-inert in the variables $u,v,m$ and $n$ and has a system of constants determined by $\mathcal{C}(\varepsilon)$. Now,
\begin{equation} \label{derivatives}
\Phi_v(u,v;m,n)=\frac{n}{q}-\frac{muN}{v^2 q} \quad \text{and} \quad \Phi_{vv}(u,v)=\frac{2muN}{v^3 q}.
\end{equation} 
Thus \eqref{secondstat} is satisfied and the critical point in the $v$ variable is located at
\begin{equation*}
\tilde{v}_{T}(u;m,n):=\tilde{v}=\sqrt{\frac{muN}{n}} \asymp N.
\end{equation*}
We apply the Theorem \ref{KPY} to the right side of \eqref{modify} with 
\begin{equation*}
Y=cMN, \quad Z=0.95N, \quad A=3, \quad \text{and} \quad R=\frac{cMN}{(q^{\varepsilon} Q)^2} \geq 1
\end{equation*}
to obtain 
\begin{equation} \label{secondint}
\overline{\widehat{G}_T \Big(\frac{m}{q},\frac{n}{q} \Big)}=\Big(\frac{N}{cM} \Big)^{\frac{1}{2}} \int_{\R}  \mathcal{W}^{1}_{T}(u;m,n)  e \Big({-cNu+\frac{2 (mnuN)^{\frac{1}{2}}}{q}}  \Big) du+O\Big( \frac{(q^{\varepsilon} Q)^6}{c^3 (MN)^2} \Big),
\end{equation} 
for some family of $X_{T}=q^{\varepsilon} Q$-inert functions $\mathcal{W}^{1}_{T}(u;m,n)$ with support the same as $w_T$ in the variables $u,m$ and $n$ and has a system of constants determined by $\mathcal{C}(\varepsilon)$. Now,
\begin{equation} \label{secondint2}
\widehat{G}_T \Big(\frac{m}{q},\frac{n}{q} \Big)=\Big(\frac{N}{cM} \Big)^{\frac{1}{2}} \int_{\R}  \mathcal{W}^{1}_{T}(u;m,n)  e \Big(cNu-\frac{2 (mnuN)^{\frac{1}{2}}}{q}  \Big) du+O \Big( \frac{(q^{ \varepsilon} Q)^6}{c^3 (MN)^2} \Big).
\end{equation}
Note in the above line we used the fact that the conjugate of an $X_T$-inert function is $X_T$-inert (we did not relabel the conjugated family of functions).

\subsection{Integration with respect to $u$}
Let 
\begin{equation*}
\tilde{\Phi}_T(u;m,n):=\Phi_T(u,\tilde{v};m,n)=cNu-\frac{2 (mnuN)^{\frac{1}{2}}}{q}.
\end{equation*}
Observe that $\tilde{\Phi}_T$ satisfies \eqref{phibd} with $Y=cMN$ with respect to the variables $u, m$ and $n$. Now,
\begin{equation*}
\tilde{\Phi}_u(u;m,n)=cN- \frac{(mnN)^{\frac{1}{2}}}{q u^{\frac{1}{2}}} \quad \text{and} \quad \tilde{\Phi}_{uu}(u;m,n)= \frac{(mnN)^{\frac{1}{2}}}{2 q u^{\frac{3}{2}}}.
\end{equation*}
Thus \eqref{secondstat} is satisfied and the critical point in the $u$ variable is located at 
\begin{equation*}
\tilde{u}_{T}(m,n):=\tilde{u}=\frac{mn}{q^2 c^2 N} \asymp M.
\end{equation*}
We apply the Theorem \ref{KPY} to the integral in \eqref{secondint} with
\begin{equation*}
Y=cMN, \quad Z=0.89M, \quad A=\frac{5}{2}, \quad \text{and} \quad R=\frac{cMN}{(q^{\varepsilon} Q)^2} \geq 1
\end{equation*}
to obtain 
\begin{align*}
\widehat{G}_T \Big(\frac{m}{q},\frac{n}{q} \Big)&=\frac{1}{c} \mathcal{W}^{2}_T(m,n) e \Big({-\frac{mn}{cq^2} }\Big)+O \Big( \frac{(q^{\varepsilon} Q)^6}{c^3 (MN)^2} \Big)+O \Big(   \Big(\frac{N}{cM} \Big)^{\frac{1}{2}} M \Big( \frac{(q^{\varepsilon} Q)^2}{cMN} \Big)^{\frac{5}{2}} \Big) \\
&=\frac{1}{c} \mathcal{W}^{2}_T(m,n) e \Big({-\frac{mn}{cq^2} }\Big)+O \Big( \frac{(q^{ \varepsilon} Q)^6}{c^3 (MN)^2} \Big),
\end{align*}
where $\mathcal{W}^{2}_T(m,n)$ denotes a family of $X_T:=q^{\varepsilon} Q$ functions in the variables $m$ and $n$ with support the same as $w_T$. That completes the proof of \eqref{equal}.

\subsection{Remarks and proof of \eqref{equal2}}
Without loss of generality we may assume $MN/8<U<8MN$, otherwise $H_{T,U} \equiv 0$. Define the function
\begin{equation*}
W_{T,U}(x,y):=W_1 \Big(\frac{x}{M} \Big) W_2 \Big(\frac{y}{N} \Big) W_3 \Big(\frac{xy}{U} \Big).
\end{equation*} 
Thus \eqref{triple3} holds and so $W_{T,U}:\mathbb{R}^2 \rightarrow \mathbb{R}$ is a family of $X_{T,U}$-inert functions with system of constants determined by $\mathcal{C}(\varepsilon)$. Thus the argument above can be followed to establish \eqref{equal2}.
\end{proof}

\section{Bilinear forms involving Kloosterman sums and oscillatory weights} \label{bilinos}
We now use the results in Sections \ref{nonstatforms} and \ref{statforms} to prove an estimate for a bilinear form involving Kloosterman sums. Let $\varepsilon>0$, $\mathbb{P}$ denote the set of prime numbers and 
\begin{multline*}
\mathcal{T}_2(\varepsilon):=\Big \{(c,q,M,N,Q) \in \R_{>0} \times \mathbb{P} \times \R_{\geq \frac{1}{2}}^2 \times \R_{\geq 1}: \\
\frac{cMN}{(q^{\varepsilon} Q)^2} \geq 1,  M \leq N \ll q, MN \ll q, \emph{ and } cN \geq 6 \Big \}.
\end{multline*} 
 
\begin{prop} \label{bilinprop}
Let $\varepsilon>0$, $T:=(c,q,M,N,Q) \in \mathcal{T}_2(\varepsilon)$ and $U \geq 1$. Furthermore, let $W_i$, $G_T$ and $H_{T,U}$ be as in \eqref{smooth}, \eqref{G} and \eqref{H} respectively. Then for all $a \in \N$ such that $(a,q)=1$ we have 
\begin{equation} \label{bilin1}
\sum_{m, n \in \Z} G_T(m,n)  \emph{Kl}_2(amn;q)  \ll (q^{\varepsilon} Q)^2 M q^{\frac{1}{2}} \Big(1+\frac{Q^4}{cM^2N} \Big),
\end{equation}
and 
\begin{equation} \label{bilin2}
\sum_{m, n \in \Z}  H_{T,U}(m,n) \emph{Kl}_2(amn;q) \ll (q^{\varepsilon} Q)^2 M q^{\frac{1}{2}} \Big(1+\frac{Q^4}{cM^2N} \Big).
\end{equation}
All implied constants depend only on $\varepsilon>0$ and $\mathcal{C}(\varepsilon)$.
\end{prop}

Observe that the bounds \eqref{bilin1} and \eqref{bilin2} perform well when $N \gg q^{\frac{1}{2}}$. Also note that for all $\varepsilon>0$ we have $\mathcal{T}_2(\varepsilon) \subset \mathcal{T}_1(\varepsilon)$, where $\mathcal{T}_1(\varepsilon)$ is defined in \eqref{T}.

\begin{proof}
To prove \eqref{bilin1}, we apply Proposition \ref{temper} with 
\begin{equation*}
G_T(x,y):= W_1 \Big(\frac{x}{M} \Big) W_2 \Big(\frac{y}{N} \Big) e(cxy) \quad  \text{and} \quad K(mn):=\text{Kl}_2(amn;q).
\end{equation*}
Observe that $\widehat{K}(0)=0$, so the first term on the right side of \eqref{voronoi} vanishes. Note that the second sum in \eqref{voronoi} is absolutely convergent by Lemma \ref{oscill2}. The second term on the right side of \eqref{voronoi} becomes
\begin{equation*} 
-\frac{1}{q^{\frac{3}{2}}}  \sum_{\substack{m,n \in \Z \\ mn \not \equiv a \mod{q} }}  \widehat{G}_T \Big(\frac{m}{q},\frac{n}{q} \Big)+\frac{q-1}{q^{\frac{3}{2}}} \sum_{\substack{m,n \in \Z \\ mn \equiv a \mod{q} }} \widehat{G}_T \Big(\frac{m}{q},\frac{n}{q} \Big).
\end{equation*}
Further simplifying, this becomes
\begin{equation} \label{setup4}
-\frac{1}{q^{\frac{3}{2}}}  \sum_{m,n \in \Z}  \widehat{G}_T \Big(\frac{m}{q},\frac{n}{q} \Big)+\frac{1}{q^{\frac{1}{2}}} \sum_{\substack{m,n \in \Z \\ mn \equiv a \mod{q} }} \widehat{G}_T \Big(\frac{m}{q},\frac{n}{q} \Big).
\end{equation}

We now estimate the summation over $m$ and $n$ in \eqref{setup4}. We break this into two cases that depend on whether the summand is non-stationary or stationary phase integral.

\subsection{Non-stationary phase contribution}
Here we estimate the contribution to \eqref{setup4} in the ranges for $m$ and $n$ occurring in the statement of Lemma \ref{oscill2}. We do this with second sum in \eqref{setup4}. A similar contribution from the first sum in \eqref{setup4} will follow from an easier computation. We use \eqref{nonstat} to bound $\widehat{G}_T$. We then use the facts that $cN \geq 6$ and $N \ll q$  to obtain the following estimates:
\begin{align}
\frac{1}{q^{\frac{1}{2}}} \sum_{|n| \geq 1.1cqM} \sum_{\substack{|m| \geq 1.1cqN  \\ mn \equiv a \pmod{q}}}  \frac{(q^{1+\varepsilon} Q)^{2+\varepsilon}}{|mn|^{1+\varepsilon}} & \ll q^{\frac{1}{2}+\varepsilon} Q^{2+\varepsilon}, \nonumber \\
 \frac{1}{q^{\frac{1}{2}}} \sum_{|n| \leq 1.1cqM} \sum_{\substack{|m| \geq 1.1cqN \\ mn \equiv a \pmod{q}}}  M N \Big(\frac{q^{1+\varepsilon} Q}{|m| M} \Big)^{2+\varepsilon} & \ll q^{\frac{1}{2}+\varepsilon} Q^{2+\varepsilon}, \nonumber \\
  \frac{1}{q^{\frac{1}{2}}} \sum_{|n| \geq 1.1cqM} \sum_{\substack{|m| \leq 1.1cqN \\ mn \equiv a \pmod{q}}}  M N \Big(\frac{q^{1+\varepsilon} Q}{|n| N} \Big)^{2+\varepsilon}  & \ll q^{\frac{1}{2}+\varepsilon} Q^{2+\varepsilon}, \nonumber \\
\frac{1}{q^{\frac{1}{2}}} \sum_{-1.1cqM  \leq n \leq 0.9 cqM} \sum_{\substack{ |m| \leq 1.1cqN \\ mn \equiv a \pmod{q} }} \frac{(q^{\varepsilon} Q)^2}{c^2 MN} & \ll q^{\frac{1}{2}+\varepsilon} Q^{2+\varepsilon} \label{finite}.
\end{align}
The case $-1.1cqN \leq m \leq 0.9 cqN$ and $|n| \leq 1.1cqM$ is handled similarly to \eqref{finite}.

\subsection{Stationary phase contribution}
Now we consider estimating the contribution to \eqref{setup4} from the ranges for $m$ and $n$ occurring in the statement of Lemma \ref{oscill}. Applying Lemma \ref{oscill} and estimates above, \eqref{setup4} becomes 
\begin{multline} \label{setup5}
\Bigg(-\frac{1}{q^{\frac{3}{2}}} \sum_{\substack{0.9cqN \leq m \leq 1.1cqN \\  0.9cqM \leq n \leq 1.1cqM }}+\frac{1}{q^{\frac{1}{2}}} \sum_{\substack{0.9cqN \leq m \leq 1.1cqN \\  0.9cqM \leq n \leq 1.1cqM \\ mn \equiv a \hspace{-0.25cm} \mod{q} }}   \Bigg) \frac{1}{c} \mathcal{W}_T(m,n)  e \Big({-\frac{mn}{cq^2}} \Big) \\
 +O \Big( \frac{q^{\frac{1}{2}} (q^{\varepsilon}Q)^6}{c M N} \Big)+O(q^{\frac{1}{2}} \big(q^{\varepsilon} Q)^2 \big),
\end{multline}
where $T \in \mathcal{T}_2(\varepsilon) \subset \mathcal{T}_1(\varepsilon)$ and $\mathcal{W}_T$ has support contained in \eqref{mnsupport}. Fourier expanding the congruence condition $\overline{a} mn \equiv 1 \pmod{q}$ as
\begin{equation*}
\mathbf{1}_{\overline{a} mn \equiv 1 \hspace{-0.25cm} \mod{q}}=\frac{1}{q} \sum_{h=0}^{q-1} e \Big( \frac{h (\overline{a} mn-1)}{q} \Big),
\end{equation*}
we see that \eqref{setup5} becomes
\begin{equation} \label{setup3}
\frac{1}{c q^{\frac{3}{2}}} \sum_{h=1}^{q-1} e \Big({-\frac{h}{q}} \Big) \sum_{\substack{0.9 cqN \leq m \leq 1.1cqN \\ 0.9cqM \leq n \leq 1.1cqM}} \mathcal{W}_T(m,n)   e \Big ( \frac{\overline{a}hmn}{q}-\frac{mn}{cq^2} \Big)+\mathcal{E},
\end{equation} 
where $\mathcal{E}$ denotes the same error terms occurring in \eqref{setup5}.  Applying Abel summation to \eqref{setup3} and using the fact that $\mathcal{W}_T$ is $q^{\varepsilon}Q$-inert with respect to $m$ and $n$, we see that \eqref{setup3} is 
\begin{equation} \label{setup6}
\ll \frac{(q^{\varepsilon} Q)^2}{c q^{\frac{3}{2}}} \sum_{h=1}^{q-1} \max_{\substack{0.9cqN \leq C_h \leq  1.1cqN  \\  0.9cqM \leq D_h \leq 1.1cqM}} \Bigg | \sum_{\substack{0.85cqN \leq m \leq C_h \\
0.85cqM \leq n \leq D_h}} e \Big({\frac{\overline{a} hmn  }{q} -\frac{mn}{cq^2}   }\Big)  \Bigg|+\mathcal{E}.
\end{equation}
We treat the long sum over $m$ as a geometric series. Letting $-q/2 \leq r \leq q/2$ be such that $\overline{a} hn \equiv r \pmod{q}$, we see that \eqref{setup6} is 
\begin{multline} \label{S2}
 \ll \frac{(q^{\varepsilon} Q)^2}{c q^{\frac{3}{2}} } \sum_{0 \leq |r| \leq q/2} \hspace{0.25cm} \sum_{0.85cqM \leq n \leq 1.1cqM} \max  \Big(  \Big \|  \frac{r}{q} -\frac{n}{cq^2}   \Big \|^{-1},   cqN \Big) \\
 \times \Bigg( \sum_{\substack{1 \leq h \leq q-1 \\ \overline{a} hn \equiv r \pmod{q} }} 1 \Bigg)+\mathcal{E},
\end{multline}  
where $\| \alpha \|$ denotes the distance from $\alpha \in \R$ to the nearest integer. 

When $r=0$ and $n \equiv 0 \pmod{q}$ in \eqref{S2}, then the innermost sum over $h$ is $q-1$. This can only happen if $1.1cM \geq 1$.  Observe that 
\begin{equation} \label{obs}
0<\frac{0.85M}{q} \leq \frac{n}{cq^2} \leq \frac{1.1M}{q}<0.2 \quad \text{for all sufficiently large primes $q$},
\end{equation}
where the last inequality follows since $1/2 \leq M \ll q^{\frac{1}{2}}$. Thus the contribution to \eqref{S2} when $r=0$ and $n \equiv 0 \pmod{q}$ is 
\begin{equation*}
\ll  \frac{(q^{\varepsilon} Q)^2}{c q^{\frac{3}{2}}} (q-1)  \sum_{l=1}^{1.1cM} \frac{cq^2}{lq} \ll q^{\frac{1}{2}} (q^{\varepsilon} Q)^2.
\end{equation*}
There is no $0<n \leq 1.1cqM$ such that $n \equiv 0 \pmod{q}$ if $1.1cM<1$.

In the summation over $h$ in \eqref{S2}, we have $r \neq 0$ if and only if $n \not \equiv 0 \pmod{q}$, otherwise the summation is empty. When $r \neq 0$, the innermost sum is 1. We have the three cases to consider in \eqref{S2}. They are
\begin{equation*}
0 \leq |r|<0.7 M, \quad  0.7 M \leq |r| \leq 1.2 M \quad \text{and} \quad 1.2 M<|r| \leq q/2.
\end{equation*}
This follows from \eqref{obs}.

The contribution to \eqref{S2} from the first case is   
\begin{equation*}
\ll \frac{(q^{\varepsilon} Q)^2}{c q^{\frac{3}{2}}} \sum_{0 \leq |r|< 0.7 M} \hspace{0.2cm}  \sum_{1 \leq n \leq 1.1cqM} \frac{cq^2}{n} \ll (q^{\varepsilon} Q)^2 q^{\frac{1}{2}} M.
\end{equation*}

The contribution to \eqref{S2} from the third case is 
\begin{equation*}
\ll \frac{(q^{\varepsilon} Q)^2}{c q^{\frac{3}{2}}} \sum_{1 \leq n \leq 1.1cqM} \hspace{0.2cm}  \sum_{1.2 M< |r| \leq q/2}  \frac{q}{r} \ll (q^{\varepsilon} Q)^2 q^{\frac{1}{2}} M.
\end{equation*}

We now consider the second case. We may replace $\| \cdot \|$ with $| \cdot |$ in \eqref{S2}. For fixed $r$ satisfying $0.7 M \leq |r| \leq 1.2 M$, the values of $n$ such that
\begin{equation*}
\Big |  \frac{r}{q} -\frac{n}{cq^2}   \Big |^{-1} \geq   cqN
\end{equation*}
are contained in an interval (denote it $B_r$) of length $2q/N$. The contribution to \eqref{S2} from such pairs $(r,n)$ is 
\begin{equation*}
\ll (q^{\varepsilon} Q)^2 q^{\frac{1}{2}} M.
\end{equation*}
We now estimate the contribution to \eqref{S2} from the pairs $(r,n)$ such that 
\begin{equation*}
\Big |  \frac{r}{q} -\frac{n}{cq^2}   \Big |^{-1} \leq   cqN.
\end{equation*}
Such a contribution is 
\begin{align}
& \ll \frac{(q^{\varepsilon} Q)^2}{c q^{\frac{3}{2}}} \sum_{0.7 M \leq |r| \leq 1.2 M} \hspace{0.2cm}  \sum_{\substack{1 \leq n \leq 1.1cqM \\ n \not \in B_r}} \Big | \frac{r}{q} -\frac{n}{cq^2}   \Big |^{-1} \nonumber \\
& \ll (q^{\varepsilon} Q)^2  q^{\frac{1}{2}} \sum_{0.7 M \leq |r| \leq 1.2 M} \hspace{0.2cm}  \sum_{\substack{1 \leq n \leq 1.1cqM \\ n \not \in B_r}} |r c q-n |^{-1}. \label{latbd}
\end{align}
Note that $n \in \Z_{\geq 1} \setminus B_r$ implies that 
\begin{equation*}
 |r c q-n | \geq \frac{q}{N} \gg 1.
\end{equation*}
Furthermore, distinct elements of
\begin{equation*}
\{|r c q-n | : n \in \Z_{\geq 1} \setminus B_r \}
\end{equation*}
have a difference that is greater than or equal to 1 in absolute value. Each value in the above set is achieved by at most two values of $n \in \Z_{\geq 1} \setminus B_r$. 
Thus the quantity in \eqref{latbd} is 
\begin{equation*}
\ll (q^{\varepsilon} Q)^2  q^{\frac{1}{2}} M.
\end{equation*}
Putting all of this together yields \eqref{bilin1}.

\subsection{Remarks and proof of \eqref{bilin2}}
Observe that \eqref{bilin2} follows by exactly the same argument as above. The main difference in the argument is that we appeal to \eqref{equal2} instead of \eqref{equal}.
\end{proof}

\section{Smoothed version of Main Theorem} \label{smoothpropsec}
The passage to proving Theorem \ref{nonsmooththm} first involves establishing an estimate for a smoothed sum of Kloosterman sums.
\begin{prop} \label{smoothprop}
Let $q$ be a prime number and $u,v \in \mathbb{N}$ be such that $(u,v)=1$ and $ v \leq q$. Let $Q \geq 1$, $\varepsilon>0$ and $W:[0,\infty] \rightarrow \R$ be a smooth function satisfying 
\begin{equation} \label{smoothhyp}
\emph{supp } W \subset [0.95,1.05] \quad \text{and} \quad W^{(j)}(x) \ll_{j,\varepsilon} (q^{\varepsilon} Q)^j \text{ for any } x \geq 0 \text{ and } j \geq 0. 
\end{equation}
 For all  
 \begin{equation} \label{Xhypstat}
 \max \{v^{\frac{6}{5}} (q^{\varepsilon}Q)^{\frac{24}{5}}, q^{\frac{3}{5}}, q^{-\frac{6}{5}+\frac{72}{5} \varepsilon} Q^{\frac{72}{5}} v^{\frac{18}{5}} \}  \ll X  \leq q
 \end{equation}
 we have 
\begin{equation} \label{smoothstat}
\sum_{\substack{p \equiv u \hspace{-0.25cm} \mod v \\ p \emph{ prime} }} W \Big( \frac{p}{X} \Big)   \emph{Kl}_2(p,q) \ll (qQ)^{\varepsilon} \big(Q^{2} X^{\frac{13}{18}} q^{\frac{1}{6}}+Q^{\frac{1}{2}} X^{\frac{2}{3}} q^{\frac{1}{4}} + Q^2 q^{\frac{11}{64}} X^{\frac{13}{16}} \big).
\end{equation} 
The implied constants depend only on $\varepsilon>0$ and $\mathcal{C}(\varepsilon)$.
\end{prop}

\begin{proof}
It is sufficient to estimate 
\begin{equation} \label{smoothv}
\mathcal{S}_{W,X}(\Lambda,\text{Kl}_2):=\sum_{n}  \mathbf{1}_{u \hspace{-0.25cm} \mod v}(n) \Lambda(n) \text{Kl}_2(n;q) W \Big(\frac{n}{X} \Big).
\end{equation}
We apply Heath Brown's identity \cite[Lemma~4.1]{FKM} and smooth partition of unity \cite[Lemme~2]{E} to \eqref{smoothv} with an integer parameter $J \geq 2$. This decomposes $\mathcal{S}_{W,X}(\Lambda,\text{Kl}_2)$ into a linear combination of $O(\log^{2J} X)$ sums with coefficients bounded by $O(\log X)$. These sums take the shape
\begin{multline*}
\Sigma(\boldsymbol{M}, \boldsymbol{N}):=\sum_{m_1, \ldots, m_J} \alpha_1(m_1) \cdots \alpha_J(m_J)  \sum_{n_1, \ldots, n_J} V_1 \Big( \frac{n_1}{N_1} \Big) \cdots V_J \Big( \frac{n_J}{N_J} \Big) \\
\times W \Big( \frac{m_1 \cdots m_J n_1 \cdots n_J}{X} \Big) \text{Kl}_2 (m_1 \cdots m_J n_1 \cdots n_J;q) \textbf{1}_{u \hspace{-0.25cm} \mod{v}}(m_1 \cdots m_J n_1 \cdots n_J), 
\end{multline*}
where
\begin{itemize}
\item $\boldsymbol{M}:=(M_1,\ldots,M_J)$ and $\boldsymbol{N}:=(N_1,\ldots,N_J)$ are $J$-tuples of parameters in $[0.95,1.05X]^{2J}$ which satisfy 
\begin{equation} \label{cond}
N_1 \geq N_2 \geq \cdots \geq N_J, \quad M_i \leq X^{1/J}, \quad M_1 \cdots M_J N_1 \cdots N_J \asymp_J X;
\end{equation}
\item the arithmetic functions $m_i \mapsto \alpha_i(m_i)$ are bounded and supported in $[0.95M_i,1.05M_i]$;
\item the smooth functions $x_i \mapsto V_i(x)$ satisfy \eqref{smoothhyp} with system of constants $\mathcal{C}(\varepsilon)$.
\end{itemize}  
Take $J=10$ and let $\mathcal{T}_2(\varepsilon)$ be as it appears in Proposition \ref{bilinprop}. In fact, any choice $J \geq 7$ will be sufficient for the combinatorial argument following \eqref{sig2}.  We now bound $\Sigma(\boldsymbol{M}, \boldsymbol{N})$ in three different ways. The three bounds will be optimised and the utility of each bound will depend on the position $(\boldsymbol{M},\boldsymbol{N})$ in the cube $[0.95,1.05X]^{20}$.

\subsection{Method 1}
Consider the vectors $(\boldsymbol{M},\boldsymbol{N})$ that satisfy
\begin{equation} \label{vecchoice}
\frac{N_1 N_2}{v (q^{\varepsilon} Q)^2} \geq 1 \quad \text{and} \quad \frac{N_1}{v} \geq 6.
\end{equation}
For such vectors, write the congruence condition in $\Sigma(\boldsymbol{M},\boldsymbol{N})$ using additive characters to obtain
\begin{multline} \label{addchar}
\Sigma(\boldsymbol{M}, \boldsymbol{N})=\frac{1}{v} \sum_{h=1}^v e \Big({-\frac{hu}{v}}  \Big) \sum_{m_1,\ldots,m_{10}} \alpha_1(m_1) \cdots \alpha_{10}(m_{10})  \sum_{n_1,\ldots,n_{10}} V_1 \Big( \frac{n_1}{N_1} \Big) \cdots V_{10} \Big( \frac{n_{10}}{N_{10}} \Big) \\
\times W \Big( \frac{m_1 \cdots m_{10} n_1 \cdots n_{10}}{X} \Big) \text{Kl}_2 (m_1 \cdots m_{10} n_1 \cdots n_{10};q) e \Big( \frac{h m_1 \cdots m_{10} n_1 \cdots n_{10}}{v} \Big).
\end{multline}
We can apply Proposition \ref{bilinprop} uniformly to the largest two smooth variables $n_1$ and $n_2$ in $\Sigma(\boldsymbol{M}, \boldsymbol{N})$. Note this process is uniform with respect to the variables $h,m_1,\ldots,m_{10},n_3,\ldots,n_{10}$. We have 
\begin{equation*}
a:=m_1 \cdots m_{10} n_3 \cdots n_{10} \quad \text{and} \quad c:=\frac{h m_1 \cdots m_{10} n_3 \cdots n_{10}}{v} \geq \frac{h}{v} \geq \frac{1}{v}. 
\end{equation*}
Observe that \eqref{cond} and \eqref{vecchoice} imply that
\begin{equation*}
\Big (\frac{h m_1 \cdots m_{10} n_3 \cdots n_{10}}{v},q,N_2,N_1,Q \Big) \in \mathcal{T}_2(\varepsilon).
\end{equation*}
Applying Proposition~\ref{bilinprop} and estimating trivially with respect to the remaining variables we obtain
\eqref{vecchoice} we obtain 
\begin{align}
\Sigma(\boldsymbol{M}, \boldsymbol{N}) & \ll \frac{X(q^{\varepsilon} Q)^2 q^{\frac{1}{2}}}{v N_1} \sum_{h=1}^v \Big(1+\frac{Q^4 v}{hN_1 N_2^2}   \Big) \nonumber \\ 
& \ll \frac{X q^{\frac{1}{2}} (q^{\varepsilon} Q)^2}{N_1} \Big(1+\frac{Q^4}{N_1 N_2^2} \Big). \label{bd1} 
\end{align}

\subsection{Methods 2 and 3}
In the second and third methods we write the congruence condition in $\Sigma(\boldsymbol{M},\boldsymbol{N})$ using multiplicative characters modulo $v$. Then 
\begin{multline} \label{mult}
\Sigma(\boldsymbol{M}, \boldsymbol{N})=\frac{1}{\phi(v)} \sum_{\chi \mod{v}} \chi(\overline{u}) \sum_{\substack{m_1, \ldots,m_{10} \\ n_1, \ldots, n_{10}}} \alpha_{\chi,1}(m_1) \cdots \alpha_{\chi,10}(m_{10}) \beta_{\chi,1} (n_1) \cdots \beta_{\chi,10}(n_{10}) \\
\times W \Big( \frac{m_1 \cdots m_{10} n_1 \cdots n_{10}}{X} \Big) \text{Kl}_2 (m_1 \cdots m_{10} n_1 \cdots n_{10};q) 
\end{multline}
where 
\begin{equation*}
\alpha_{\chi,j}(m_j):=\chi(m_j) \alpha(m_j)  \quad \text{and} \quad \beta_{\chi,j}(n_j):=\chi(n_j) V \Big( \frac{n_j}{N_j} \Big).
\end{equation*}

For the second method, apply Theorem \ref{FKM} to obtain
\begin{equation} \label{bd2}
\Sigma(\boldsymbol{M}, \boldsymbol{N}) \ll q^{\varepsilon} Q^{\frac{1}{2}} X \Big(\frac{1}{M^{\frac{1}{2}}}+\frac{q^{\frac{1}{4}}}{(X/M)^{\frac{1}{2}}}  \Big),
\end{equation}
for any factorisation
\begin{equation} \label{factor}
X \asymp M_1 \cdots M_{10} N_1 \cdots N_{10}=M N.
\end{equation}

For the third method, apply Theorem \ref{KMS}. Since $X \gg q^{\frac{3}{5}}$ by hypothesis and \eqref{factor} holds, we have $MN>q^{\frac{1}{4}}$ for all sufficiently large primes $q$.  We also have $MN \ll q$, so $MN<q^{\frac{5}{4}}$ for all sufficiently large primes $q$. We obtain
\begin{equation} \label{bd3}
\Sigma(\boldsymbol{M}, \boldsymbol{N}) \ll (q^{\varepsilon} Q)^2 X \Big(\frac{1}{M^{\frac{1}{2}}}+X^{-\frac{3}{16}} q^{\frac{11}{64}}   \Big)
\end{equation}
for any factorisation \eqref{factor} such that 
\begin{equation} \label{speccond}
1 \leq  M \leq Nq^{\frac{1}{4}}.
\end{equation}

\subsection{Optimisation}
Introduce real numbers $\kappa, \theta, x,\mu_i$ and $\nu_j$ with $1 \leq i,j \leq 10$, defined by
\begin{equation} \label{paramstar}
Q=q^{\kappa}, \quad v=q^{\theta}, \quad X=q^x, \quad M_i=q^{\mu_i}, \quad \text{and} \quad N_j=q^{\nu_j}.
\end{equation}
We set 
\begin{equation*}
(\boldsymbol{m},\boldsymbol{n}):=(\mu_1,\ldots,\mu_{10},\nu_1,\ldots,\nu_{10}) \in [0,x]^{20}.
\end{equation*}
The conditions \eqref{cond} when $J=10$ can be restated as 
\begin{equation} \label{restate}
 \sum_i \mu_i+\sum_j \mu_j=x \leq 1, \quad \mu_i \leq \frac{x}{10}, \quad \text{and} \quad \nu_1 \geq \nu_2 \geq \cdots \geq \nu_{10}.
\end{equation}
Let $\beta$ be a parameter that satisfies 
\begin{equation} \label{betabd}
\max \big (4 \kappa+\theta+4 \varepsilon, \frac{5x}{12}+\varepsilon \big ) \leq \beta \leq  \frac{5x}{6}.
\end{equation}
We choose $\beta$ at a later point in \eqref{choicebeta}. Note that the hypothesis $X \geq v^{\frac{6}{5}} (q^{\varepsilon}Q)^{\frac{24}{5}}$ ensures $\beta$ is well-defined. 

\subsubsection{Large $N_1$}
Suppose $\nu_1 \geq \beta$. Then \eqref{vecchoice} holds. Thus we may apply \eqref{bd1} and the fact $N_2 \geq 1/2$ to obtain 
\begin{equation} \label{sig1}
\Sigma(\boldsymbol{M}, \boldsymbol{N}) \ll (qQ)^{\varepsilon} Q^2 X q^{\frac{1}{2}-\beta}.
\end{equation}

\subsubsection{Small $N_1$}
Suppose $\nu_1 \leq {\beta}$. Note that the estimate in \eqref{smoothstat} is trivial for $x \leq 3/4$, so we may assume $x>3/4$. Define the function $\eta(\boldsymbol{m},\boldsymbol{n})$ by  
\begin{equation*}
\eta(\boldsymbol{m},\boldsymbol{n}):=\min \Big( \frac{\sigma}{2}, \frac{x-\sigma}{2}-\frac{1}{4}  \Big)-\frac{\kappa}{2},
\end{equation*}
where $\sigma$ ranges over all possible sub-sums of the $\mu_i$ and $\nu_j$ for $1 \leq i,j \leq 10$, i.e., over the sums 
\begin{equation*}
\sigma=\sum_{i \in \mathcal{I}} \mu_i+\sum_{j \in \mathcal{J}} \nu_j,
\end{equation*}
for $\mathcal{I}$ and $\mathcal{J}$ ranging over all possible subsets of $\{1,\ldots,10\}$. Thus \eqref{bd2} can be rewritten as 
\begin{equation} \label{combd}
\Sigma(\boldsymbol{M}, \boldsymbol{N}) \ll (qQ)^{\varepsilon} q^{-\eta(\boldsymbol{m},\boldsymbol{n})} X.
\end{equation}
Let
\begin{equation*}
\mathcal{J}_x:= \Big [\frac{x}{6}-\varepsilon,\frac{x}{3}+\varepsilon \Big].
\end{equation*}

If $(\boldsymbol{m},\boldsymbol{n})$ contains a subsum $\sigma \in \mathcal{J}_x$, then we have 
\begin{equation*}
\eta(\boldsymbol{m},\boldsymbol{n}) \geq \min \Big( \frac{x/6}{2},\frac{x-x/3}{2}-\frac{1}{4}  \Big)-\frac{\kappa}{2}-\frac{\varepsilon}{2}.
\end{equation*}
This becomes 
\begin{equation*}
\eta(\boldsymbol{m},\boldsymbol{n}) \geq \frac{x}{3}-\frac{1}{4}-\frac{\kappa}{2}-\frac{\varepsilon}{2}.
\end{equation*}
In this case \eqref{combd} asserts 
\begin{equation} \label{sig2}
\Sigma(\boldsymbol{M},\boldsymbol{N}) \ll (qQ)^{\varepsilon} Q^{\frac{1}{2}} X^{\frac{2}{3}} q^{\frac{1}{4}}.
\end{equation}

If $(\boldsymbol{m},\boldsymbol{n})$ contains no sub-sum $\sigma \in \mathcal{J}_x$, then the sum of all $\mu_i$ and $\nu_j$ which are less than $x/6-\varepsilon$ is also less than $x/6-\varepsilon$. This follows from the inequality $2(x/6-\varepsilon)<x/3+\varepsilon$. Thus \eqref{restate} implies all $\mu_i<x/6-\varepsilon$ (this is where the fact $J \geq 7$ is used). Thus some of the $\nu_j$ must be greater than $x/3+\varepsilon$. Since $3(x/3+\varepsilon)>1$, at most two of the $\nu_j$ ($\nu_1$ or $\nu_1$ and $\nu_2$) are greater than $x/3+\varepsilon$.  Thus 
\begin{equation*}
\nu_1+\nu_2 \geq x-\big(\frac{x}{6}-\varepsilon \big)=\frac{5x}{6}+\varepsilon.
\end{equation*}
Now, the condition $\nu_1 \leq \beta$ and \eqref{cond} implies that 
\begin{equation} \label{chain}
 \frac{5x}{6}-\beta+\varepsilon \leq \nu_2 \leq \nu_1 \leq \beta.
\end{equation}
Observe that \eqref{chain} is well-defined because \eqref{betabd} holds. We may apply \eqref{bd3} to bound $\Sigma(\boldsymbol{M},\boldsymbol{N})$ with choices
\begin{equation*}
M:=N_2 \quad \text{and} \quad N:=\frac{X}{N_2},
\end{equation*}
since $q^{\frac{3}{4}} \leq X \leq q$ and $1 \leq M \leq Nq^{\frac{1}{4}}$ (the last condition is trivial). We obtain
\begin{equation} \label{sig3}
\Sigma(\boldsymbol{M}, \boldsymbol{N}) \ll (q^{\varepsilon} Q)^2 \big( X^{\frac{7}{12}} q^{\frac{\beta}{2}}  + q^{\frac{11}{64}} X^{\frac{13}{16}} \big).
\end{equation}

\subsection{Combining the bounds}
We now sum over all $(\boldsymbol{m},\boldsymbol{n})$. Combining \eqref{sig1}, \eqref{sig2} and \eqref{sig3} yields
\begin{equation} \label{total}
\mathcal{S}_{W,X}(\Lambda,\text{Kl}_2) \ll  (qQ)^{\varepsilon} \big( Q^2 X q^{\frac{1}{2}-\beta}+Q^{\frac{1}{2}} X^{\frac{2}{3}} q^{\frac{1}{4}}+Q^2 X^{\frac{7}{12}} q^{\frac{\beta}{2}} + Q^2 q^{\frac{11}{64}} X^{\frac{13}{16}} \big).
\end{equation}
Recall the parameters defined in \eqref{paramstar}. To balance the first and third terms of \eqref{total} we set
\begin{equation} \label{choicebeta}
\beta:=\frac{1}{3}+\frac{5}{18}x.
\end{equation}
Such a $\beta$ satisfies \eqref{betabd} as long as 
\begin{equation*}
x \geq \max  \Big \{\frac{3}{5},  \frac{18}{5} \theta+\frac{72 \kappa}{5}+\frac{72}{5} \varepsilon-\frac{6}{5} \Big \},
\end{equation*}
which is guaranteed by \eqref{Xhypstat}. Thus \eqref{total} becomes
\begin{equation} \label{total2}
\mathcal{S}_{W,X}(\Lambda,\text{Kl}_2) \ll (qQ)^{\varepsilon} \big(Q^2 X^{\frac{13}{18}} q^{\frac{1}{6}}+Q^{\frac{1}{2}} X^{\frac{2}{3}} q^{\frac{1}{4}} + Q^2 q^{\frac{11}{64}} X^{\frac{13}{16}} \big).
\end{equation}
\end{proof}

\section{Proof of Main Theorem} \label{mainthms}
We now apply the result of Section \ref{smoothpropsec} with an appropriate choice of parameters to prove the main theorem.
\begin{proof}[Proof of Theorem \ref{nonsmooththm}]
Note that \eqref{tradstat} is trivial for $X \ll q^{\frac{11}{12}}$, so we may assume that $X \gg  q^{\frac{11}{12}}$. First recall that $\varepsilon>0$ is small, $q \in \mathbb{P}$, $X \geq 1$ and $v \in \N$ are such that
\begin{equation} \label{paramhyp}
q^{\frac{11}{12}} \ll X \leq q \quad \text{and} \quad 1 \leq v \leq q^{\frac{1}{100}}.
\end{equation}
Furthermore, let
\begin{equation} \label{Qhyp}
40 \leq Q \leq X^{\frac{1}{16}}.
\end{equation}
It is clear that \eqref{paramhyp} and \eqref{Qhyp} imply that \eqref{Xhypstat} is satisfied. Let $\Delta:=Q^{-1}$ and suppose $W:[0,\infty] \rightarrow \R$ is smooth and satisfies 
\begin{align*}
\text{supp} \subset \big[0.975-\Delta, 1.025+ \Delta \big],  & \quad 0 \leq W \leq 1,  \quad  W(x)=1 \quad \text{for} \quad 0.975 \leq x \leq 1.025 \\
\quad \text{and} \quad & x^j W^{(j)}(x)  \ll_{\varepsilon,j} Q^j
\end{align*}
with system of constants $\mathcal{C}(\varepsilon)$. Applying Proposition \ref{smoothprop} we obtain 
\begin{align}
\sum_{\substack{0.975X<p \leq 1.025 X \\ p \equiv u \hspace{-0.15cm} \mod{v} }} \text{Kl}_2(p;q) & \ll \Delta X+1+ \Bigg | \sum_{\substack{p \equiv u \hspace{-0.25cm} \mod v \\ p \text{ prime} }} W \Big( \frac{p}{X} \Big) \text{Kl}_2(p;q) \Bigg |  \nonumber \\
& \ll Q^{-1} X+(qQ)^{\varepsilon} \big(Q^2 X^{\frac{13}{18}} q^{\frac{1}{6}}+Q^{\frac{1}{2}} X^{\frac{2}{3}} q^{\frac{1}{4}}+ Q^2 q^{\frac{11}{64}} X^{\frac{13}{16}} \big) \label{balance}. 
\end{align}
We choose $Q:=q^{-\frac{11}{192}} X^{\frac{1}{16}} \leq X^{\frac{1}{16}}$ to balance the first and fourth terms of \eqref{balance}. Since $X \gg q^{\frac{11}{12}}$ we have $Q \geq 40$. Thus 
\begin{align*}
\sum_{\substack{0.975 X<p \leq 1.025 X  \\ p \equiv u \hspace{-0.15cm} \mod{v}}} \text{Kl}_2(p;q) & \ll q^{\varepsilon} \big( q^{\frac{11}{192}} X^{\frac{15}{16}}+q^{\frac{5}{96}} X^{\frac{61}{72}}+q^{\frac{85}{384}} X^{\frac{67}{96}} \big) \\
& \ll  q^{\frac{11}{192}+\varepsilon} X^{\frac{15}{16}},
\end{align*}
where the last line follows from $q^{\frac{11}{12}} \ll X \leq q$. Summing dyadically yields the result.
\end{proof}

\section{Acknowledgments}
The authors thank Scott Ahlgren and the referee for their meticulous comments on the manuscript.

\end{document}